\documentclass[10pt]{amsart}

\usepackage{amssymb}
\usepackage[all]{xy}
\usepackage{enumerate}

\newcommand{\rar}[1]{\stackrel{#1}{\longrightarrow}}

\newcommand{\bA}{{\mathbb A}}

\newcommand{\bF}{{\mathbb F}}

\newcommand{\bQ}{{\mathbb Q}}

\newcommand{\bZ}{{\mathbb Z}}

\newcommand{\cA}{{\mathcal A}}

\newcommand{\cD}{{\mathcal D}}

\newcommand{\cO}{{\mathcal O}}

\newcommand{\cQ}{{\mathcal Q}}

\newcommand{\cW}{{\mathcal W}}

\newcommand{\td}{\tilde{d}}
\newcommand{\tf}{\tilde{f}}

\newcommand{\tv}{\tilde{v}}
\newcommand{\tw}{\widetilde{w}}
\newcommand{\tx}{\widetilde{x}}
\newcommand{\ty}{\widetilde{y}}
\newcommand{\tz}{\widetilde{z}}

\newcommand{\rx}{\widetilde{x^{[2]}}}
\newcommand{\ry}{\widetilde{y^{[2]}}}
\newcommand{\rz}{\widetilde{z^{[2]}}}

\newcommand{\nc}{\newcommand}

\nc\wh{\widehat}

\nc\on{\operatorname}

\nc\Gr{\on{Gr}}

\nc\Fl{\on{Fl}}

\newcommand{\limto}{{\displaystyle\lim_{\longrightarrow}}}
\newcommand{\rightlim}{\mathop{\limto}}

\newcommand{\leftlim}{\mathop{\displaystyle\lim_{\longleftarrow}}}
\newcommand{\limfromn}{\leftlim\limits_{\raise3pt\hbox{$n$}}}
\newcommand{\limton}{\rightlim\limits_{\raise3pt\hbox{$n$}}}

\newcommand{\rightlimit}[1]{\mathop{\lim\limits_{\longrightarrow}}\limits%
                    _{\raise3pt\hbox{$\scriptstyle #1$}}}

\newcommand{\leftlimit}[1]{\mathop{\lim\limits_{\longleftarrow}}\limits%
                    _{\raise3pt\hbox{$\scriptstyle #1$}}}

\newcommand{\iso}{\buildrel{\sim}\over{\longrightarrow}}
\newcommand{\mono}{\hookrightarrow}

\DeclareMathOperator{\Spec}{{Spec}}

\makeatletter

\newcommand{\Rmnum}[1]{\expandafter\@slowromancap\romannumeral #1@}
\makeatother

\newtheorem{Th}{Theorem}
\newtheorem{pr}{Proposition}[section]
\newtheorem{lm}[pr]{Lemma}

\newtheorem{cor}[pr]{Corollary}

\newtheorem{example}{Example}

\theoremstyle{definition}

\newtheorem{rem}[pr]{Remark}
\newtheorem{cl}[pr]{Claim}

\numberwithin{equation}{section}

\begin{document}

\title[Differential operators on a smooth variety over $\bZ/p^n\bZ$]{On the center of the ring of differential operators on a smooth variety over $\bZ/p^n\bZ$.}

\author{Allen Stewart \quad Vadim Vologodsky}

\address{Department of Mathematics, University of Oregon, Eugene, OR, 97403, USA}
\email{allens@uoregon.edu, vvologod@uoregon.edu}

%and
%\affilnum{2}Complete Second Author Address}

% Address / e-mail address of corresponding author
%\correspdetails{corr.email@math.edu}



%\date{}

\begin{abstract} We compute the center of the ring of PD differential operators on a smooth variety over $\bZ/p^n\bZ$ confirming a conjecture of Kaledin. More generally, given an associative algebra $A_0$ over 
$\bF_p$ and its flat deformation $A_n$ over $\bZ/p^{n+1}\bZ$   
we prove that under a certain non-degeneracy condition the center of $A_n$ is isomorphic to the ring of length $n+1$ Witt vectors over the center of $A_0$. 
\end{abstract}

\maketitle

\section{Introduction}  
\subsection{}
Let $X_n$ be a smooth scheme over the spectrum $S_n$ of the ring of length $n+1$ Witt vectors $W_{n+1}(k)$ over a perfect field $k$ of characteristic $p$, $X_0$ its special fiber over $k$,
 and let $D_{X_n}=D_{X_n/S_n}$ be the sheaf of PD differential operators on $X_n$. We prove in Theorem \ref{3} that
the center $Z(D_{X_n})$ of   $D_{X_n}$ is canonically isomorphic to the ring of  Witt vectors $W_{n+1}(S^{\cdot}T_{X_0})$  over the symmetric algebra of the tangent sheaf of $X_0$. For $n=0$ we recover the classical
isomorphism (see, {\it e.g} \cite{bmr})
\begin{equation}\label{pcurv}
Z(D_{X_0})\simeq S^{\cdot}T_{X_0}
\end{equation}
 given by the $p$-curvature map. The general result was conjectured by Kaledin. For $p \ne 2$ he even proposed a construction of the map  $$W_{n+1}(S^{\cdot}T_{X_0})\to Z(D_{X_n}).$$
\subsection{} In fact, we prove a more general result. Let $A_n $ be a flat associative algebra over $W_{n+1}(k)$, $n>0$. Set 
$$A_i= A_n \otimes _{W_{n+1}(k)} W_{i+1}(k), \quad  0\leq i \leq n,$$
 and let $Z(A_i)$ be the center of $A_i$. The first order deformation $A_1$ yields a natural biderivation on $Z(A_0)$ (\S \ref{def}, (\ref{poi})):
 $$\{,\}:  Z(A_0)\otimes_k Z(A_0) \to Z(A_0).$$
 We shall say that the deformation $A_n$ of $A_0$ is non-degenerate if  $spec\, Z(A_0)$ is smooth over $k$ and the biderivation, $\{,\}$, is associated with a non-degenerate bivector 
 field, $\mu\in  \bigwedge^2T_{Z(A_0)}$, on $spec\, Z(A_0)$.
 
 %\footnote{If $n>1$, then $\{,\}$ also satisfies the Jacobi identity {\it i.e.}, defines a  Poisson structure on $Z(A_0)$.} $$
 %We shall say that the deformation $A_n$ of $A_0$ is non-degenerate if the null space of $\{,\}$ coincides with the image $Z(A_0)^p \subset Z(A_0)$ of the Frobenius map on $Z(A_0)$.\footnote{Note that our non-degeneracy condition depends only on the first order
%deformation $A_1$.} For example, this condition is satisfied if 
 %$spec\, Z(A_0)$ is smooth over $k$ and the biderivation $\{,\}$ is associated with a symplectic form $\omega \in \Omega^1_{spec\, Z(A_0)}$. 
 If $z$ is an element of $Z(A_0)$ and $\tilde z\in A_n$ is a lifting
of $z$, then, for every $0 \leq i \leq n$,  the element $p^{i} \tilde z^{p^{n-i}}\subset A_n$ is central and does not depend on the choice of $\tilde z$. We prove in Theorem \ref{p} that for $p \ne 2$ the map
\begin{equation}\label{map}
\phi_n: W_{n+1}(Z(A_0)) \to Z(A_n)
\end{equation}
defined by the formula
\begin{equation}\label{dmapin}
\phi_n((z_1, z_2, \cdots z_{n+1})) = \sum_{i=0}^n p^{i} \tilde z_{i+1}^{p^{n-i}}
\end{equation} 
is a homomorphism of rings
and if the deformation, $A_n$, of $A_0$ is non-degenerate then $\phi_n$ is an isomorphism. Note that the left-hand 
side of (\ref{map}) depends only on the algebra $A_0$ and not on the deformation $A_n$.

\subsection{} For $p=2$ the map $\phi_n$ given by formula (\ref{dmapin}) is neither additive nor multiplicative\footnote{We are grateful to Pierre Berthelot for pointing out this problem.}.  However, 
under the additional assumption that the deformation, $A_n$, is non-degenerate  and the differential  $2$-form $\omega=\mu^{-1} \in \Omega^2_{Z(A_0)}$ associated with $\{,\}$ is exact,
$\omega=d\eta$, we can correct our map (\ref{dmapin})  as follows.  The Poisson algebra $Z(A_0)$ has a restricted structure in the sense of Bezrukavnikov-Kaledin (\cite{bk}):  if $z\in Z(A_0)$ and $t_z$ is the corresponding Hamiltonian vector field on  $spec\, Z(A_0)$ {\it i.e.}, $dz=i_{t_z}\omega$,  we set
 \begin{equation}\label{rs}
z^{[p]}= L^{p-1}_{t_z} i_{t_z} \eta - i_{t_z^{[p]}}\eta \in Z(A_0),
\end{equation}
 where $t_z^{[p]}\in T_{Z(A_0)}$ is the $p$-th power in the restricted Lie algebra of vector fields and $L_{t_z}$ is the Lie derivative.
 For $p=2$ we define 
 \begin{equation}\label{dmaptwo}
 \phi_n((z_1, z_2, \cdots z_{n+1})) = \sum_{i=0}^{n-1} 2^{i} (\tilde z_{i+1}^2 +2\widetilde{z_{i+1}^{[2]}}) ^{2^{n-i-1}} + 2^n \tilde z_{n+1}.
\end{equation} 
 We prove in Theorem \ref{2} that the map $\phi_n: W_{n+1}(Z(A_0)) \to Z(A_n)$ given by the above formula is an isomorphism of rings.
 
\subsection{} According to an observation of Kanel-Belov and Kontsevich (\cite{kk}), for every smooth scheme, $X_n$, over $S_n$ the bivector field on $Z(D_{X_0}) \simeq S^{\cdot}T_{X_0} $ induced by the deformation $X_n$ equals, up to sign, the bivector field on 
 $S^{\cdot}T_{X_0}$ induced by the canonical symplectic structure on the cotangent bundle ${\bf T}^*_{X_0}$. In particular, the former bivector field is non-degenerate and the associated differential form has a canonical primitive $\eta\in \Omega^1_{{\bf T}^*_{X_0}}$. Thus, as a corollary of the above results, we find a canonical isomorphism of sheaves of rings 
 \begin{equation}\label{mdiff}
\phi_n:  W_{n+1}(S^{\cdot}T_{X_0}) \simeq Z(D_{X_n}).
 \end{equation}
 %\footnote{In fact, formula (\ref{rs}) defines a restricted Poisson structure on $Z(A_0)$ for every prime $p$, not only for $p=2$.} 

\subsection{Acknowledgments.}  This paper owes its existence to Dima Kaledin who explained to the second author his conjecture on the center of the ring of differential operators.  Also,  would like to thank Pierre Berthelot, Arthur Ogus, and Victor Ostrik  for helpful conversations related to the subject of this paper.

 \bigskip

\section{Main result. Odd characteristic case.} 
\subsection{}\label{def} 
 Let $R_n$ be a commutative algebra flat over $\bZ/p^{n+1}\bZ$, $n>0$.  For $0\leq m \leq n$ we set 
$$R_m= R_n \otimes _{\bZ/p^{n+1}\bZ} \bZ/p^{m+1}\bZ.$$
 By a level $n$ deformation of a flat associative $R_0$-algebra, $A_0$, we mean a flat associative $R_n$-algebra, $A_n$, together with an isomorphism
$A_n\otimes_{R_n}R_0\cong A_0$. Given such $A_n$ we denote by $A_m$ the corresponding algebra over $R_m$.
We will write  
\[
A_m\rar{r}A_{m-1}, \quad
r(x)= x \mod p^m
\] 
for the reduction homomorphism. The preimage of $x\in pA_m\subset A_m$ under the isomorphism
\[
 A_{m-1} \rar{p} pA_m 
\]
is denoted by $\frac{1}{p}x$.
We will write $Z_m=Z(A_m)$ for the center of $A_m$. The following Lemma is straightforward. 
   \begin{lm}\label{lm1} Let $x\in Z_i$, $y\in Z_j$, $0\leq i\leq j \leq n $, and let $\tilde x, \tilde y\in A_n$ be liftings of $x$ and $y$ respectfully. Then  
   \begin{enumerate}
\item $[\tx,\ty]\equiv 0\mod p^{j+1}$
\item $[\tx,\ty] \mod p^{i+j+2} \in Z_{i+j+1}$.
\item The element  $\frac{1}{p^{j+1}}[\tx,\ty] \mod p^{i+1} \in Z_i$ is independent of the choice of liftings $\tx$, $\ty$.
\item  The $R_n$-linear map 
 \begin{equation}\label{genpoi}
Z_i \otimes _{R_n} Z_j \to Z_i, \quad x\otimes y  \mapsto \frac{1}{p^{j+1}}[\tx,\ty] \mod p^{i+1}
\end{equation}
is a derivation with respect to the first argument. 
\item  The element $(\tx)^p \mod p^{i+2} $ lies in $ Z_{i+1}$  and is independent of  the choice of the lifting $\tx$.
\end{enumerate}
\end{lm}
In the case $i=j=0$ the map (\ref{genpoi}) deserves a special notation:
 \begin{equation}\label{poi}
\{,\}:  Z(A_0)\otimes_{R_0} Z(A_0) \to Z(A_0), \quad \{x, y\} = \frac{1}{p}[\tx,\ty] \mod p.
\end{equation}
By part (4) of the Lemma $\{,\}$ is a derivation with respect to each argument. 
We also remark that if $n>1$ the map $\{,\}$ satisfies the Jacobi identity. This follows by dividing the identity
$$[\tx, [\ty, \tz]]+ [\tz, [\tx, \ty]] + [\ty, [\tz, \tx]]=0$$
by $p^2$ and reducing the result modulo $p$.
Thus, if $n>1$ the bracket $\{,\}$ defines a Poisson structure on $Z_0$.  

We shall say that $A_n$ is a non-degenerate deformation of $A_0$ if $A_0$ is a smooth $R_0$-algebra and the map  $\{,\}$ is associated with
a non-degenerate bivector field $\mu\in  \bigwedge^2T_{Z_0/R_0}$.
 That is
\[
\{x,y\}=< \mu, dx \otimes dy>,
\]
for every $x,y \in Z_0$.  Viewing $\mu$ as $Z_0$-linear isomorphism $T^*_{Z_0/R_0}\to T_{Z_0/R_0}$ and taking its inverse $T_{Z_0/R_0}\to T^*_{Z_0/R_0}$ we obtain a differential $2$-form,
$\omega= \mu^{-1} \in  \Omega^2_{Z_0/R_0}$.  The form $\omega$ is closed if and only if the bracket $\{,\}$ is Poisson.

We remark that our non-degeneracy condition depends only on the reduction of $A_n$ modulo $p^2$. 

\subsection{}
Let $W_{m+1}(Z_0)$ be the ring of length $m+1$ Witt vectors of $Z_0$. For $0\leq m \leq n$ we define a map 
$$\phi_m: W_{m+1}(Z_0)\to A_{m}$$
by 
\begin{equation}\label{defin}
\phi_m (z_1,\ldots,z_{m+1})=\sum_{i=0}^{m}p^i\tz_{i+1}^{p^{m-i}},
\end{equation}
where $\tz_i $ is a lifting of $z_i\in Z_0$ in $A_m$.
 \begin{cl}
The map $\phi$ is well defined and the image of $\phi$ is contained in $Z_{m}$:
 \begin{equation}\label{phinetwo}
 \phi_m: W_{m+1}(Z_0)\to Z_{m}
 \end{equation}
\end{cl}
\begin{proof}
If $\tz_{i+1}$ and $\tz_{i+1}'$ are liftings of $z_{i+1}$, then by part (5) of Lemma \ref{lm1},  we have 
\[
(\tz_{i+1})^{p^{m-i}} \equiv (\tz_{i+1}')^{p^{m-i}} \mod   p^{m-i+1} \in Z_{m-i} 
\]
 which implies that 
\[
p^i (\tz_{i+1})^{p^{m-i}} = p^i (\tz_{i+1}')^{p^{m-i}}  \in Z_{m}. 
\]
\end{proof}
%Applying the above construction to $A_n=R_n$ we find a map
 %\begin{equation}\label{base}
 % W_{m+1}(R_0)\to R_{m},
 %\end{equation}
%which is a ring homomorphism as follows from the commutative diagram
%\begin{equation}
%\xymatrix{
% W_{m+1}(R_m)  \ar[d]_{}  \ar[r]^{w_{m+1}} & R_m\\
 %W_{m+1}(R_0)  \ar[ur]_{\phi_m} & \\
%}
%\end{equation}
%where $w_{m+1}$ is the homomorphism given by the $(m+1)$-th Witt polynomial and the vertical arrow is the reduction map.
%We shall denote by $W_{m+1}(Z_0)\otimes _{W_{m+1}(R_0)} R_m$ the tensor product with respect to the  $W_{m+1}(R_0)$-module structure on 

\begin{Th}\label{p}
Suppose $p\ne 2$.  Then, for every flat associative algebra $A_n$ over $R_n$ and every $0\leq m \leq n$, the map $\phi_m$ is a ring homomorphism.  If, in addition, the deformation $A_n$ is non-degenerate
the homomorphism
\begin{equation}\label{mm}
\Phi_m: W_{m+1}(Z_0)\otimes _{W_{m+1}(R_0)} R_m \to Z_{m}, \quad \Phi_m(z\otimes a)= a\phi_m(z)
 \end{equation}
is an isomorphism. Here the $W_{m+1}(R_0)$-module structure on $W_{m+1}(Z_0)$ is induced by the embedding 
$R_0\mono Z_0$ and  $W_{m+1}(R_0)$-module structure on $R_m$ is given by the homomorphism $ W_{m+1}(R_0) \to R_m$.
\end{Th}
The proof of this theorem occupies the rest of this section. 
\subsection{}\label{wittvectors}  We begin with some general remarks on Witt vectors.
It is well known (see, {\it e.g.}  \cite{m}, \S 26) and easy that the polynomials
\begin{equation}
\psi_i(x,y)=\frac{(x+y)^{p^{i-1}}-(x^p+y^p)^{p^{i-2}}}{p^{i-1}}, \quad i>1
\end{equation}
\begin{equation}
\psi_1(x,y)=x+y
\end{equation}
have integral coefficients and satisfy the recursive formula, 
\begin{equation}\label{rec.f}
 \psi_{i+1}(x,y)= \sum_{j=1}^p \binom{p}{j} p^{j(i-1)-i} \psi_i(x,y)^j  (x^p+y^p)^{(p-j)p^{i-2}}
\end{equation}
for $i>1$. 
We claim that for every commutative ring $Z$ and every $x,y\in Z$, one has the following equation in $W_n(Z)$,
\begin{equation}\label{add}
\underline{x+y}=\sum_{i=0}^{n-1}V^i\psi_{i+1}(\underline{x},\underline{y}).
\end{equation}
Here we write $ \underline{x}=(x,0,\ldots,0)$ for the Teichm\"uller representative of $x$ in $W_m(Z)$ and $V$ for the Verschiebung operator $W_{m}(Z)\to W_{m+1}(Z)$.
Indeed, it suffices to check this identity for $Z=\bZ[x, y]$. In this case the ghost map $W_n(Z) \to Z^n$ given by the Witt polynomials $\cW_m$ is an injective homomorphism. Thus, it is enough to check that ghost coordinates of both sides of  (\ref{add}) are equal. We have,
\[
\cW_m(\underline{x+y})=(x+y)^{p^{m-1}}= \sum_{i=0}^{m-1}p^i\psi_{i+1}(x^{p^{m-i-1}},y^{p^{m-i-1}})=
\]
\[
\sum_{i=0}^{m-1}p^i\psi_{i+1}(\cW_{m-i}(\underline x),  \cW_{m-i}(\underline y))= \cW_m\left(\sum_{i=0}^{n-1}V^i\psi_{i+1}(\underline{x},\underline{y})\right),
\]
where we used that $\cW_i \circ V = p \cW_{i-1}$. This proves (\ref{add}).

We are interested in describing ring homomorphisms from $W_n(Z)$ to a given ring.
\begin{lm}\label{lem}
Let $Z_0,Z_1,\ldots,Z_{n-1}$ be commutative rings. Suppose we are given two families of maps $\chi^{(i)}:Z_0\to Z_{i}$ and $\pi:Z_i\to Z_{i+1}$, $i=0,\ldots,n-1$, such that the following conditions hold:
\begin{enumerate}[(a)]
\item $\chi^{(i)}$ is multiplicative and $\pi$ is additive,
\item For any $x,y\in Z_0$, $p\pi(xy)=\pi(x)\pi(y)$ and if $0\le i \le  m\le n-1$ then 
\[
\chi^{(m)}(x)\pi^i\chi^{(m-i)}(y)=\pi^i\chi^{(m-i)}(x^{p^i}y)
\]
\item For any $x,y\in Z_0$ and $0\le m \le n-1$,
\[
\chi^{(m)}(x+y)=\sum_{i=0}^{m}\pi^{i}\psi_{i+1}(\chi^{(m-i)}(x),\chi^{(m-i)}(y)))
\]
\end{enumerate}
Then the maps $\varphi_m:W_{m+1}(Z_0)\to Z_{m}$ defined by
\begin{equation}\label{defpsi}
 \varphi_m(z_1,\ldots,z_{m+1})=\sum_{i=0}^{m}\pi^{i}\chi^{(m-i)}(z_{i+1})
\end{equation}
are ring homomorphisms and
\begin{equation}\label{vpi} 
\pi\varphi_{m-1}=\varphi_mV.
\end{equation}
\end{lm}
\begin{proof}
Formula (\ref{vpi}) is clear. We prove that $\varphi_m$ is a ring homomorphism using induction on $m$.
We have that $\varphi_0(z)=\chi^{(0)}(z)$. By assumption $\chi^{(0)}$ is multiplicative and using property (c) it follows that $\chi^{(0)}$ is additive. Hence $\varphi_0$ is a ring homomorphism. Now suppose $\varphi_l$ is a ring homomorphism for $l<m$. Let $x=(x_1,\ldots,x_{m+1})\in W_{m+1}(Z_0)$,  $x'=(x_2,\ldots,x_{m+1})\in W_{m}(Z_0)$,  and let $w$ be another element of $W_{m}(Z_0)$. Then, using the induction assumption and (\ref{vpi}), we have 
\begin{align*}
\varphi_m(x+Vw)&=\varphi_m(\underline{x_1}+V(x'+w))\\
&=\varphi_m(\underline{x_1})+\varphi_m(V(x'+w))\\
&=\varphi_m(\underline{x_1})+\pi\varphi_{m-1}(x'+w)\\
&=\varphi_m(\underline{x_1})+\varphi_{m}(Vx')+\varphi_m(Vw)\\
&=\varphi_m(x)+\varphi(Vw)
\end{align*}
Thus, for any $x\in W_{m+1}(Z_0)$ and $w \in W_{m}(Z_0)$, we have 
\begin{equation}\label{vaddprop}
\varphi_m(x+Vw)=\varphi_m(x)+\varphi_m(Vw)
\end{equation}
This implies that it suffices to check additivity of $\varphi_m$ on Witt vectors of the form $\underline{z}$. Adjusting equation (\ref{add}) we have 
\begin{equation}
\underline{z}+\underline{z'}=\underline{z+z'}-\sum_{i=1}^{m}V^i\psi_{i+1}(\underline{z},\underline{z'})
\end{equation}
Therefore, using (\ref{vpi}) as well as  (\ref{vaddprop}) and induction we see
\begin{align*}
\varphi_m(\underline{z}+\underline{z'})&=\varphi_m\left(\underline{z+z'}-\sum_{i=1}^{m}V^i\psi_{i+1}(\underline{z},\underline{z'})\right)\\
&=\varphi_m(\underline{z+z'})-\varphi_m\left(\sum_{i=1}^{m}V^i\psi_{i+1}(\underline{z},\underline{z'})\right)\\
&=\varphi_m(\underline{z+z'})-\sum_{i=1}^{m}\pi^i\psi_{i+1}(\varphi_{m-i}(\underline{z}),\varphi_{m-i}(\underline{z'}))\\
&=\varphi_m(\underline{z+z'})-\sum_{i=1}^{m}\pi^i\psi_{i+1}(\chi^{(m-i)}(\underline{z}),\chi^{(m-i)}(\underline{z'}))\\
&=\chi^{(m)}(z+z')-\sum_{i=1}^{m}\pi^i\psi_{i+1}(\chi^{(m-i)}(\underline{z}),\chi^{(m-i)}(\underline{z'}))
\end{align*}
Hence, by property (c) it follows that $\varphi_m(\underline{z}+\underline{z'})=\chi^{(m)}(z)+\chi^{(m)}(z')$, which implies $\varphi_m$ is additive. 

Since $\varphi_m$ is additive 
%we have 
%\begin{equation}
%\varphi_m\left(\sum_{i}V^i\underline{z_i}\cdot\sum_{j}V^j\underline{z'_j}\right)=
%\varphi_m\left(\sum_{i,j}V^i\underline{z_i}V^j\underline{z'_j}\right)=
%\sum_{i,j}\varphi_m\left(V^i\underline{z_i}V^j\underline{z'_j}\right)
%\end{equation}
%Thus 
it suffices to check multiplicativity on Witt vectors of the form $V^i\underline{z}$. We have $V^i\underline{z}\cdot V^j\underline{z'}=p^iV^i(\underline{z}\cdot V^{j-i}\underline{z'})$. Notice that $p\pi(xy)=\pi(x)\pi(y)$ implies that $p^i\pi^i(xy)=\pi^i(x)\pi^i(y)$. If $i\ne 0$ then using this fact along with the facts $\varphi_m$ is additive and $\varphi_mV=\pi\varphi_{m-1}$ by induction it follows that
\begin{align*}
\varphi_m(V^i\underline{z}V^j\underline{z'})&=\varphi_m(p^iV^i(\underline{z}\cdot V^{j-i}\underline{z'}))\\
&=p^i\pi^i(\varphi_{m-i}(\underline{z})\cdot \varphi_{m-i}(V^{j-i}\underline{z}'))\\
&=\pi^i\varphi_{m-i}(\underline{z})\cdot \pi^i\varphi_{m-i}(V^{j-i}\underline{z'})\\
&=\varphi_m(V^i\underline{z})\cdot \varphi_m(V^j\underline{z'})
\end{align*}
If $i=0$ then we have $\underline{z}\cdot V^j(\underline{z'})=V^j(\underline{z^{p^j}z'})$ and using property (b) it follows that
\begin{align*}
\varphi_m(\underline{z}\cdot V^j\underline{z'})&=\varphi_m(V^j(\underline{z^{p^j}z'}))\\
&=\pi^j\varphi_{m-j}(\underline{z^{p^j}z'})\\
&=\pi^j\chi^{(m-j)}(z^{p^j}z')\\
&=\chi^{(m)}(z)\cdot \pi^j\chi^{(m-j)}(z')\\
&=\varphi_m(\underline{z})\cdot \varphi_{m}(V^j\underline{z'})
\end{align*}
\end{proof}

\subsection{} To show that the map $\phi_m$ in Theorem \ref{p} is a ring homomorphism we check that for $Z_i=Z(A_i)$ the maps $\chi^{(i)}:Z_0\to Z_{i}$ and $\pi:Z_i\to Z_{i+1}$ defined by $\chi^{(i)}(z)=\tz^{p^i}$ and $\pi(z)=pz$ satisfy the conditions of Lemma \ref{lem}. The only assertions that deserve proof
are the multiplicativity of  $\chi^{(m)}$ and property  (c) which is implied by the following identity:
\begin{equation}\label{bnf}
 p^i \psi_{i+1}(\tx^{p^j}, \ty^{p^j})=  (\tx^{p^j}+ \ty^{p^j})^{p^{i}} -  (\tx^{p^{j+1}}+ \ty^{p^{j+1}})^{p^{i-1}},
\end{equation}
for every $\tx,\ty \in A_m$ that are central modulo $p$ and every $i,j$ with $i+j=m$. 

Suppose elements $\tx,\ty \in A_{m+1}$  are central modulo $p^{m+1}$ {\it i.e.}, their reductions in $A_m$ lie in $Z_m$, then using the fact that $[\tx,\ty]\in Z_{m+1}$, we have
\[
(\tx\ty)^p=\tx^p\ty^p-\binom{p}{2}x^{p-1}y^{p-1}[\tx,\ty].
\]
As $\binom{p}{2}$ is divisible by $p$ (here we use that $p\ne 2$) and $[\tx,\ty]\equiv 0 \mod p^{m+1}$ it follows that
\begin{equation}\label{chimult}
(\tx\ty)^p=\tx^p\ty^p.
\end{equation}
The multiplicativity of  $\chi^{(m)}$ is derived from (\ref{chimult}) by induction. 

To prove (\ref{bnf}) we need the following.
\begin{pr}\label{binom}
Suppose $p\ne 2$ and $\tx,\ty\in A_{m+1}$ are central modulo $p^{m+1}$. Then the following equation holds 
\[
(\tx+\ty)^p=\sum_{i=0}^{p} \binom{p}{i} \tx^{i}\ty^{p-i}.
\]
\end{pr}
\begin{proof}
Let $L$ be the Lie algebra over $\bZ[1/2]\subset \bQ$ which is free as a $\bZ[1/2]$-module with basis $X,Y$ and $[X,Y]$ such that $[X,Y]$ is in the center of $L$. Denote the universal enveloping algebra of $L$ by $U_L$.  By Poincar\'e-Birkhoff-Witt theorem  (\cite{b} \S 1.2.7)
$U_L$ is a free $\bZ[1/2]$-module.
Now let $L_{\bQ}=L\otimes_{\bZ[1/2]}\bQ$, and let $U_{L_{\bQ}}$ be the universal enveloping algebra of $L_{\bQ}$.  The algebra $U_{L_{\bQ}}$ has a unique grading such that $deg\, X= deg\, Y =1$. Let $U_{L_{\bQ}}^i$ be the $i^{th}$ graded piece and let 
\[
\widehat{U_{L_{\bQ}}}=\prod_{i=0}^{\infty}U_{L_{\bQ}}^i
\]
be the completion of $U_{L_{\bQ}}$. Since $[X,Y]$ is in the center of $L_{\bQ}$ by the Campbell-Hausdorff formula (e.g. \cite{b} \S 2.6.4) the following equation holds in $\widehat{U_{L_{\bQ}}}$:
\[
e^{X}e^{Y}=e^{X+Y+[X,Y]/2}
\]
By equating $p^{th}$ graded terms of the above equation, we have 
\[
\sum_{i=0}^{p}\frac{X^i}{i!}\frac{Y^{p-i}}{(p-i)!}=\frac{(X+Y)^p}{p!}+(p-1)\frac{(X+Y)^{p-2}[X,Y]}{2\cdot (p-1)!}+\cdots
\] 
By multiplying both side by $p!$, we have in $U_{L_{\bQ}}$
\begin{equation}\label{pbw}
\sum_{i=0}^{p}\binom{p}{i}X^iY^{p-i}=(X+Y)^p+p(p-1)\frac{(X+Y)^{p-2}[X,Y]}{2}+\cdots
\end{equation}
The above equation has coefficients in $\bZ[1/2]$ and so, since $U_L$ is flat over $\bZ[1/2]$, holds in $U_L$.  We remark that all the coefficients in  (\ref{pbw}) except for the first one are divisible by $p$ in $\bZ[1/2]$. 

Now consider the homomorphism $U_L\to A_{m+1}$ that takes $X$ and $Y$ to $\tx$ and $\ty$ respectfully. As $[\tx,\ty]\equiv 0 \mod p^{m+1}$ the Proposition follows from (\ref{pbw}).
 \end{proof}
 When $i=1$ formula (\ref{bnf}) follows directly for the Proposition. For $i>1$ we prove (\ref{bnf}) by induction on $i$.  We have
 $$
 (\tx^{p^j}+ \ty^{p^j})^{p^{i}}= ((\tx^{p^{j+1}}+ \ty^{p^{j+1}})^{p^{i-2}} +  p^{i-1} \psi_{i}(\tx^{p^j}, \ty^{p^j}))^p$$
 $$= (\tx^{p^{j+1}}+ \ty^{p^{j+1}})^{p^{i-1}} +  p^i \sum_{l=1}^p \binom{p}{l} p^{l(i-1)-i} \psi_i(\tx^{p^j},\ty^{p^j})^l  (\tx^{p^{j+1}}+\ty^{p^{j+1}})^{(p-l)p^{i-2}}$$ 
and (\ref{bnf}) follows from recursive formula (\ref{rec.f}).
\subsection{}\label{isomorphism} It remains to prove that if the deformation $A_n$ is non-degenerate the homomorphism $\Phi_m: W_{m+1}(Z_0)\otimes _{W_{m+1}(R_0)} R_m \to Z_{m}$ is an isomorphism. When $m=0$, $\Phi_1$ is clearly an isomorphism. Now assume that $\Phi_l$ is an isomorphism for all $0\le l<m+1$. 
For a positive integer $i$ we denote by 
$$F^{i}_{Z_0/R_0}: Z_0^{(i)}=Z_0 \otimes_{F^{i}_{R_0}}R_0\to Z_0$$ 
the $i^{th}$ iterate of the relative Frobenius map.
 
In order to show that $\Phi_{m+1}$ is an isomorphism we need the following result.
\begin{pr}\label{keylm}
Let $r(Z_{m+1})\subset Z_0$ be the image of the reduction map, then $r(Z_{m+1})=Im(F^{m+1}_{Z_0/R_0})$.
\end{pr}
\begin{proof}
We maintain the assumption that $\Phi_l$ is an isomorphism for all $0\le l<m+1$. The containment $r(Z_{m+1})\supset Im(F^{m+1}_{Z_0/R_0})$ is clear  by Lemma \ref{lm1} (5). 

Let $y\in Z_{m}$, $x\in Z_0$,  and let $\tx, \ty \in A_n$  be liftings of $x$ and $y$ respectfully. Then the element $[\ty,\tx]/p^{m+1} \mod p\in Z_0$ is independent of the liftings.
Moreover, by Lemma \ref{lm1} (4) the map
\begin{equation}\label{eqpi1}
\Pi_y: Z_0\to Z_0, \quad  x\mapsto [\ty,\tx]/p^{m+1} \mod p\, 
\end{equation}
is a $R_0$-linear derivation, $\Pi_y \in T_{Z_0/R_0}$. Identifying $T_{Z_0/R_0}$ with $\Omega^1_{Z_0/R_0}$, we get a linear map
\begin{equation}\label{eqpi2}
\Pi:Z_{m}\to \Omega^1_{Z_0/R_0}, \quad y \mapsto  i_{\Pi_y} \mu^{-1}. 
\end{equation}
Note that if $y\in Z_m$ is the reduction $\mod p^{m+1}$ of some element $\ty\in Z_{m+1}$ then  $\Pi(y)=0$.
\begin{lm}\label{lm7} The image of an element $(z_1,\ldots,z_{m+1})\otimes a \in  W_{m+1}(Z_0)\otimes R_{m}$ under the composition
$$S=\Pi \circ \Phi_m : W_{m+1}(Z_0)\otimes R_{m} \iso Z_{m} \rar{}  \Omega^1_{Z_0/R_0}$$
is given by the formula
\begin{equation}\label{serremap}
S ((z_1,\ldots,z_{m+1})\otimes a)= \overline{a}\sum_{i=0}^{m}z_{i+1}^{p^{m-i}-1}dz_{i+1},
\end{equation}
where $ \overline{a}\in R_0$ is the reduction of  $a$ modulo $p$.
\end{lm}
\begin{proof}
We start with the following.
\begin{cl}
If $x\in Z_0$ and $z\in Z_i$ and $\tx,\tz\in A_n$ are any liftings, we have 
\begin{equation}\label{e9}
[\tz^p,\tx]\equiv p\tz^{p-1}[\tz,\tx]\mod p^{i+3}
\end{equation}
\end{cl}
Indeed,  
\begin{align*}
[\tz^p,\tx]&=\sum_{j=0}^{p-1}\tz^{p-j-1}[\tz,\tx]\tz^j\\
&=\sum_{j=0}^{p-1}\left(\tz^{p-1}[\tz,\tx]-\tz^{p-j-1}[\tz^j,[\tz,\tx]]\right)
\end{align*}
Since $z\in Z_i$, we have $[\tz, [\tz,\tx]] \mod p^{i+3}\in Z_{i+2}$. Thus $\tz^{p-j-1}[\tz^j,[\tz,\tx]]=j\tz^{p-2}[\tz, [\tz,\tx]] \mod p^{i+3}$ and 
\begin{align*}
\sum_{j=0}^{p-1}\left(\tz^{p-1}[\tz,\tx]-\tz^{p-j-1}[\tz^j,[\tz,\tx]]\right)&\equiv \sum_{j=0}^{p-1}\left(\tz^{p-1}[\tz,\tx]-j\tz^{p-2}[\tz,[\tz,\tx]]\right)\\
&\equiv p\tz^{p-1}[\tz,\tx]-\binom{p}{2}\tz^{p-2}[\tz,[\tz,\tx]] \mod p^{i+3}
\end{align*}
Since $p\ne 2$ we have $\binom{p}{2}\equiv 0\mod p$ and, thus, $\binom{p}{2}[\tz,[\tz,\tx]]\equiv 0\mod p^{i+3}$, which gives (\ref{e9}).

The Claim above implies, by induction, that for every $i\geq 0$, $x,z\in Z_0$, we have  
\[
[\tz^{p^i},\tx]\equiv p^i\tz^{p^i-1}[\tz,\tx]\mod p^{i+2}.
\] 
Thus, we conclude
\[ 
[\sum_{i=0}^{m}p^i\tz_{i+1}^{p^{m-i}}, \tx ] \equiv p^{m+1}\sum_{i=0}^{m}z_{i+1}^{p^{m-i}-1}[\tz_{i+1}, \tx] \mod p^{m+2},
\]
which implies the desired result.
\end{proof}
The following result will also be used in the next section.
\begin{lm}\label{serre} Let $p$ be a prime number (not necessarily odd), $Z_0$ a smooth $R_0$-algebra, and let 
$S : W_{m+1}(Z_0)\otimes R_{m}  \rar{}  \Omega^1_{Z_0/R_0}$ be the morphism (the ``Serre morphism'') defined by formula (\ref{serremap}).
 If $z\in ker\, S$ then the image of $z$ under the map 
$$\alpha: W_{m+1}(Z_0)\otimes R_{m} \to Z_0^{(m)}, \quad (z_1,\ldots,z_{m+1})\otimes a\mapsto z_1\otimes \overline{a}$$
is contained in the image of the relative Frobenius map $F_{Z_0^{(m)}/R_0}: Z_0^{(m+1)} \to Z_0^{(m)}$.
\[
\xymatrix{
ker\, S \ar@{^{(}->}[r]\ar[d]_{} & W_{m+1}(Z_0)\otimes R_{m}  \ar[d]_{\alpha}\\
 Z_0^{(m+1)}    \ar@{^{(}->}[r]_{F_{Z_0^{(m)}/R_0}  }&  Z_0^{(m)} \\
}
\]
\end{lm}
\begin{proof}
Recall (see, {\it e.g.} \cite{i}) that for every smooth $R_0$-algebra $Z_0$ we have the Cartier isomorphism
$$C^{-1}:\Omega^1_{Z_0^{(1)}/R_0}= \Omega^1_{Z_0/R_0} \otimes_{F_{R_0}} R_0  \iso H^1(\Omega^{\cdot}_{Z_0/R_0})\subset \Omega^{1}_{Z_0/R_0}/d(Z_0)$$
                         $$ xdy \otimes \overline{a} \mapsto  \overline{a} x^py^{p-1}dy, \quad xdy\in \Omega^1_{Z_0/R_0}, \, \overline{a}\in R_0. $$
More generally, for each positive integer $i$ we shall define a $R_0$-module $D_i$ together with a $R_0$-linear map
$$C^{-i}:\Omega^1_{Z_0^{(i)}/R_0}\to D_i.$$
$D_1$ is just the quotient of $\Omega^{1}_{Z_0/R_0}$ by the subspace $d(Z_0)$ of exact forms. Assuming that $D_i$ and $C^{-i}$ are already defined
we define $D_{i+1}$ to be the quotient of $D_i$ by   $C^{-i}(d(Z_0^{(i)}))$ and $C^{-i-1}$ to be the composition
$$\Omega^1_{Z_0^{(i+1)}/R_0} \rar{C^{-1}} \Omega^1_{Z_0^{(i)}/R_0}/d(Z_0^{(i)})\rar{C^{-i}}  D_{i+1}.$$
  As $C^{-1}$ is injective $C^{-i}$ is injective as well. By construction $D_i$ is a quotient of $ \Omega^{1}_{Z_0/R_0}$; we denote by $\beta:  \Omega^{1}_{Z_0/R_0}\to D_i$ the projection.
  We have the commutative diagram:
\[
\xymatrix{
 W_{m+1}(Z_0)\otimes R_{m} \ar[r]^-{S}\ar[d]_{\alpha} & \Omega_{Z_0/R_0}^1\ar[d]_{\beta}\\
 Z_0^{(m)} \ar[r]_{C^{-m}d}&  D_m\\
}
\]
If $z\in ker\, S$ then  $C^{-m}d(\alpha(z))=0$ and, thus, $d(\alpha(z))=0$. Therefore $\alpha(z)$ lies in the image of the relative Frobenius map.
  \end{proof}
 Now we can finish the proof of Proposition \ref{keylm}. As we have observed above if $y \in Z_m$ is the reduction $\mod p^{m+1}$ of an element $\ty\in Z_{m+1}$ then  $\Pi(y)=0$.
Consider the following commutative diagram
  \begin{equation}\label{diagram}
\xymatrix{
Z_{m+1}\ar[d]\ar[r] & Z_{m}\ar[d]_{\Phi_m^{-1}} \ar[r]^{r} & Z_0 & &\\
ker \, S \ar@{^{(}->}[r] & R_m\otimes W_{m+1}(Z_0) \ar[r]^-{\alpha} & Z_0^{(m)} \ar[u]_{F^m_{Z_0/R_0}}& &Z_0^{(m+1)} \ar@{_{(}->}[ll]_{F_{Z_0^{(m)}/R_0}}
}
\end{equation}
 By Lemma \ref{serre} the map $\alpha:  ker \, S \rar{\alpha} Z_0^{(m)}$ factors through $\xymatrix{Z_0^{(m+1)}\ar[rr]^{F_{Z_0^{(m)}/R_0}}&& Z_0^{(m)}}$. Thus the reduction map $r$ factors through 
 $\xymatrix{Z_0^{(m+1)}\ar[rr]^{F^{m+1}_{Z_0/R_0}}&& Z_0}$.
  \end{proof}
  
  Now we finish the proof that $\Phi_{m+1}$ is an isomorphism. Consider the following commutative diagram
\[
\xymatrix{
\cdots \ar[r] & W_{m+1}(Z_0)\otimes R_{m+1}\ar[r]^{V\otimes Id}\ar[d]_f & W_{m+2}(Z_0)\otimes R_{m+1} \ar[r]^-{g}\ar[d]_{\Phi_{m+1}} & Z_0\otimes R_{m+1} \ar[r]\ar[d]_{\overline{\Phi_{m+1}}}& 0\\
0\ar[r] & Z_{m+1}\cap pA_{m+1}\ar[r] & Z_{m+1} \ar[r]\ar[dr]_{r} & \frac{Z_{m+1}}{Z_{m+1}\cap pA_{m+1}}\ar[r]\ar@{^{(}->}[d] & 0\\
& & & Z_0
}
\]
Here $f: W_{m+1}(Z_0)\otimes R_{m+1}\simeq W_{m+1}(Z_0)\otimes R_{m}\to Z_{m+1}\cap pA_{m+1}$ equals $p\Phi_{m}$ which is an isomorphism by the induction assumption,
$g$ is induced by the homomorphism $W_{m+2}(Z_0) \to Z_0$ that takes a Witt vector to its first coordinate, and, finally, the morphism $\overline{\Phi_{m+1}}: Z_0\otimes R_{m+1}\simeq Z_0^{(m+1)} \to \frac{Z_{m+1}}{Z_{m+1}\cap pA_{m+1}} \subset Z_0 $ is equal to $F_{Z_0/R_0}^{m+1}$. By  Proposition   \ref{keylm} $\overline{\Phi_{m+1}}$ is an isomorphism. It follows that $\Phi_{m+1}$ is an isomorphism as well.

%\subsection{}\label{char2}
\section{Main result. Characteristic $2$ case.} 
\subsection{}\label{restr}
 Throughout this section $R_n$ is a commutative algebra flat over $\bZ/2^{n+1}\bZ$, $n>0$, and $A_n$ is a flat associative $R_n$-algebra. We will also assume that the deformation $A_n$ of $A_0$ is non-degenerate and denote by $\omega \in \Omega^2_{Z_0/R_0}$ the corresponding non-degenerate $2$-form.
 Though the map $W_{n+1}(Z_0)\to Z_{n}$  defined by equation (\ref{defin}) is neither additive nor multiplicative  we explain in this section that, if $\omega$ is exact,  formula  (\ref{defin}) can be corrected to yield an isomorphism of $R_n$-algebras, 
 $$W_{n+1}(Z_0)\otimes_{W_{n+1}(R_0)} R_n \iso Z_{n}.$$  
 Our construction depends on the choice of a primitive $\eta\in \Omega^1_{Z_0/R_0}$, $\omega=d\eta$. Define a map
 \begin{equation}\label{rst1}
Z_0 \to Z_0, \quad z \mapsto z^{[2]}
\end{equation}
by the formula
 \begin{equation}\label{rst2}
z^{[2]}= L_{t_z} i_{t_z} \eta - i_{t_z^{[2]}}\eta \in Z(A_0),
\end{equation}
 where $t_z\in T_{Z_0/R_0}$ is the Hamiltonian vector field corresponding to $z$  {\it i.e.}, $dz=i_{t_z}\omega$, $t_z^{[2]}\in T_{Z_0/R_0}$ is its square in the restricted Lie algebra of vector fields, and $L_{t_z}$ is the Lie derivative.
\begin{lm} For every $x,y \in Z_0$, we have
\begin{equation}\label{qadd}
(x+y)^{[2]} - x^{[2]}-y^{[2]}  =\{x,y\},
\end{equation}
\begin{equation}\label{qad}
\{x^{[2]}, y\}=\{x,\{x,y\}\},
\end{equation}
\begin{equation}\label{qmult}
(xy)^{[2]}=y^2x^{[2]}+x^2y^{[2]}+xy\{x,y\}.
\end{equation}
\end{lm}
\begin{proof}
Define a map $\cQ: T_{Z_0/R_0} \to Z_0$ by the formula,
\begin{equation}\label{qrefi}
\cQ(\theta)= L_{\theta} i_{\theta} \eta - i_{\theta^{[2]}}\eta , \quad \theta \in T_{Z_0/R_0}.
\end{equation}
Then, for every $\theta_1, \theta_2\in T_{Z_0/R_0}$, using the identity $(\theta_1+ \theta_2)^{[2]}= \theta_1^{[2]} + \theta_2^{[2]} +[\theta_1, \theta_2]$ and Cartan's formula,  we find
\begin{equation}\label{qform1}
  \cQ(\theta_1+ \theta_2)- \cQ(\theta_1)-\cQ(\theta_2) = L_{\theta_1} i_{\theta_2} \eta +  L_{\theta_2} i_{\theta_1} \eta - i_{[\theta_1, \theta_2]}\eta   =  i_{\theta_1} i_{\theta_2} \omega. 
  \end{equation}
Using $i_{t_x} i_{t_y} \omega= \{x,y\}$ equation (\ref{qadd}) follows. Next, for every $z\in Z_0$,  using the identity $(z \theta)^{[2]}= z^2\theta^{[2]}+z(L_\theta z) \theta$, one has
\begin{equation}\label{qform2}
\cQ(z\theta)= z L_{\theta} (z i_{\theta} \eta) - z^2 i_{\theta^{[2]}} \eta - z( L_{\theta} z)  i_{\theta}\eta= z^2 \cQ(\theta).
\end{equation}
Thus, we conclude 
$$(xy)^{[2]}= \cQ(xt_y +yt_x)= x^2\cQ(t_y)+y^2\cQ(t_x) + xy i_{t_x} i_{t_y} \omega,$$
which proves (\ref{qmult}).
Finally,  for (\ref{qad}) it suffices to check that $t_x^{[2]}=t_{x^{[2]}}$ or,  equivalently,  that $i_{t_x^{[2]}}\omega =  i_{t_{x^{[2]}}} \omega$  . We have,
\begin{equation}\label{eq7}
i_{t_{x^{[2]}}} \omega= dx^{[2]}= d(L_{t_x} i_{t_x} \eta - i_{t_x^{[2]}}\eta)= - L_{t_x} i_{t_x} \omega +L^2_{t_x}\eta +  i_{t_x^{[2]}}\omega - L_{t_x^{[2]}} \eta .
\end{equation}
Since $L_{t_x} i_{t_x} \omega = d\{x,x\}=0$ and $L^2_{t_x}\eta = L_{t_x^{[2]}} \eta$  the right-hand side of (\ref{eq7}) equals $ i_{t_x^{[2]}}\omega$ as required.
 
\end{proof} 

\begin{rem}
Equations (\ref{qform1}) and  (\ref{qform2}) show that the quadratic form $\cQ$ on the $Z_0$-module  $T_{Z_0/R_0}$ is a quadratic refinement of the symmetric form $\omega$. In fact, for every smooth $R_0$-algebra $Z_0$ in characteristic $2$ one can define a refined de Rham complex
$$(S^{\cdot}\Omega^1_{Z_0/R_0}, \td)= Z_0 \to \Omega^1_{Z_0/R_0} \to S^{2}\Omega^1_{Z_0/R_0} \to \cdots $$
to be the initial object in the category of commutative DG algebras $\cA$ over $R_0$ equipped with a homomorphism $Z_0\to \cA$. By the universal property the DG algebra  $(S^{\cdot}\Omega^1_{Z_0/R_0}, \td)$ maps to the de Rham DG algebra $(\bigwedge^{\cdot}\Omega^1_{Z_0/R_0}, \td)$.
The quadratic form $\cQ\in S^{2}\Omega^1_{Z_0/R_0}$ is identified with $\td \eta$.
\end{rem}
\begin{rem} In (\cite{bk}, Def. 1.8), Bezrukavnikov and Kaledin introduced the notion of restricted Poisson algebra in characteristic $p$. If $p=2$ a restricted Poisson algebra is just a Poisson algebra $Z_0$ over $R_0$  together with a map 
$Z_0 \to Z_0$, $z \mapsto z^{[2]}$ satisfying equations (\ref{qadd}), (\ref{qad}) and (\ref{qmult}). In fact, for any $p$, every smooth Poisson algebra associated with an exact symplectic form $\omega=d\eta$ has a resticted structure given by the formula (cf. \cite{bk}, Th. 1.12)
$$z^{[p]}= L^{p-1}_{t_z} i_{t_z} \eta - i_{t_z^{[p]}}\eta. $$
\end{rem}

The main result of this section is the following theorem.
\begin{Th}\label{2}
Let $R_n$ is a flat commutative algebra over $\bZ/2^{n+1}\bZ$, and $A_n$ be a flat associative algebra over $R_n$ such that  the center $Z_0$ is smooth over $R_0$ and the bracket $\{,\}: Z_0\otimes Z_0 \to Z_0$ is associated with an exact 
symplectic form $\omega=d \eta$.
Then, for every $0\leq m \leq n$, the map 
$$\phi_m:W_{m+1}(Z_0)\to Z_m$$
given by 
\[
\phi_m(z_1,\ldots,z_{m+1})=\sum_{i=0}^{m-1}2^i(\tz_{i+1}^2+2\widetilde{z_{i+1}^{[2]}})^{2^{m-i-1}}+2^m\tz_{m+1}
\]
is a ring homomorphism. Moreover,  the induced homomorphism 
\begin{equation}\label{mm2}
\Phi_m: W_{m+1}(Z_0)\otimes _{W_{m+1}(R_0)} R_m \to Z_{m}, \quad \Phi_m(z\otimes a)= a\phi_m(z)
 \end{equation}
 is an isomorphism.
\end{Th}
The proof of this theorem occupies the rest of this section.
\subsection{}We will prove that $\phi_{m+1}$ is a homomorphism and $\Phi_{m+1}$ is an isomorphism simultaneously. It is clear that $\phi_0= \Phi_0$ are isomorphisms. 

Now consider the general case. We will assume that $\phi_l$ is a homomorphism and $\Phi_l$ is an isomorphism for all $0\le l\le m$ for everything that follows. We will need the following result.
\begin{lm}\label{trivcomm}
If $x\in Z_1$, $y\in Z_i$ with $1\le i\le m$, and $\tx,\ty\in A_{n}$ are any liftings then 
\[
[\tx,\ty]\equiv 0 \mod 2^{i+2}.
\]
\end{lm}
\begin{proof}
We may assume that $m\geq 1$ (otherwise, the statement is empty). Then, by our induction hypothesis the map $\Phi_1: W_{2}(Z_0)\otimes R_1 \to Z_{1}$ is surjective. Thus, it suffices to check the Lemma for $\tx$ of the form
    $\tx= \tw^2 + 2\tv$ with $\tw, \tv\in A_n$ central modulo $2$.  We have
    \begin{align*}
[\tx,\ty]&=[\widetilde{w}^{2}+2\widetilde{v},\ty]\\
&=\tw [\widetilde{w}, \ty]  +  [\widetilde{w}, \ty] \tw    +  2[\widetilde{v},\ty] \equiv 0 \mod 2^{i+2}
\end{align*}
since the elements $[\widetilde{w}, \ty], [\widetilde{v}, \ty]$ are central modulo $2^{i+1}$. 
\end{proof}
\begin{cor}\label{2powmult}
If $x,y\in Z_i$ with $1\le i\le m$ and $\tx,\ty\in A_{i+1}$ are any liftings then, we have 
\[
(\tx\ty)^{2}=\tx^{2}\ty^{2},
\]
\[ 
(\tx+\ty)^2\equiv \tx^2+2\tx\ty+\ty^2.
\]
\end{cor}

Let $\pi:Z_i\to Z_{i+1}$ be given by $\pi(z)=2z$ and $\chi^{(i)}:Z_0\to Z_i$ be defined by $\chi^{(i)}(z)=(\tz^2+2\rz)^{2^{i-1}}$ for $0<i\le m+1$ with $\chi^{(0)}(z)=z$. We will use Lemma \ref{lem} to show that $\phi_{m+1}$ is a homomorphism.  Let us check that $\chi^{(i)}$ is multiplicative.
For $i=1$ using formula (\ref{qmult}), we have 
$$\chi^{(1)}(xy)= \tx^2\ty^2- \tx[\tx,\ty]\ty +2(xy)^{[2]}=  \tx^2\ty^2- 2xy\{x,y\} +2(x^2 y^{[2]} +y^2x^{[2]} +xy\{x,y\})$$
$$= \chi^{(1)}(x)\chi^{(1)}(y).$$
The general case follows now from Corollary \ref{2powmult}.

 Now we show property (b) of Lemma \ref{lem} which is the following identity
 \[
 (\tx^2+2\rx)^{2^{j}}(\ty^2+2\ry)^{2^{j-i}}\equiv ((\tx^{2^i}\ty)^2+2\widetilde{(x^{2^i}y)^{[2]}})^{2^{j-i}} \mod 2^{j-i+2}
\] 
for every $0\leq i\leq j\leq m$.
 If $i=0$ then this is equivalent to $\chi^{(j)}$ being multiplicative. 
 Assume that $i>0$.  By Corollary \ref{2powmult} it follows that 
\[
(\tx^2+2\rx)^{2^{j}}(\ty^2+2\ry)^{2^{j-i}}\equiv ((\tx^2+2\rx)^{2^{i}}(\ty^2+2\ry))^{2^{j-i}}\mod 2^{j-i+2}
\]
and thus to show the desired result it suffices to check that 
\[
(\tx^2+2x^{[2]})^{2^{i}}(\ty^2+2y^{[2]})\equiv (\tx^{2^i}\ty)^2+2(x^{2^i}y)^{[2]}\mod 4.
\]
Now $(\tx^2+2x^{[2]})^{2^{i}}\equiv \tx^{2^{i+1}}\mod 4$. Thus, we have
 \begin{align*}
(\tx^2+2x^{[2]})^{2^{i}}(\ty^2+2y^{[2]})&\equiv \tx^{2^{i+1}}(\ty^2+2y^{[2]}) \\
 &\equiv (\tx^{2^i}\ty)^2+2x^{2^{i+1}}y^{[2]}\\
 &\equiv (\tx^{2^i}\ty)^2+2(x^{2^i}y)^{[2]}\mod 4,
\end{align*}
where the last congruence is implied by  (\ref{qmult}).

Let us check property (c) of Lemma \ref{lem}. It suffices to prove the following claim.
\begin{cl}\label{2homclaim}
\begin{enumerate}[(a)]
\item If $x,y\in Z_{m+2-j}$ with $1<j< m+2$ and $\tx, \ty \in A_{m+1}$ are any liftings  then 
\[
2^{j-1}\psi_j(x,y)=(\tx+\ty)^{2^{j-1}}-(\tx^2+\ty^2)^{2^{j-2}}\mod 2^{m+2}
\]
\item If $x,y\in Z_0$ then for $1<j\le m+1$, we have  
\[
2^{j-1}\psi_j(x,y)=\chi^{(j-1)}(x+y)-(\widetilde{\chi^{(1)}(x)}+\widetilde{\chi^{(1)}(y)})^{2^{j-2}}\mod 2^j
\]
\end{enumerate}
\end{cl}
If $j=2$ then part (a)  follows from Corollary \ref{2powmult}. Now assume property (a) holds for $2\le l< j$ then the recursive
definition of $\psi$, equation (\ref{rec.f}), and the induction hypothesis imply 
$$
2^{j-1}\psi_{j}(x,y)\equiv 2((\tx+\ty)^{2^{j-2}}-(\tx^2+\ty^2)^{2^{j-3}})(\tx^2+\ty^2)^{2^{j-3}}+((\tx+\ty)^{2^{j-2}}-(\tx^2+\ty^2)^{2^{j-3}})^2
$$
modulo $2^{m+2}$. Applying Corollary \ref{2powmult} to the right hand side, we have the desired result. Now let's consider property (b). If $j=2$ then by using (\ref{qadd}), we have 
\begin{align*}
\chi^{(1)}(x+y)-(\chi^{(1)}(x)+\chi^{(1)}(y))&\equiv 2xy\mod 4
\end{align*} 
which is the desired result. Now assuming that property (b) holds for $2\le l< j$, we find
$$\chi^{(j-1)}(x+y)= (\widetilde{\chi^{(j-2)}(x+y)})^2= ( (\widetilde{\chi^{(1)}(x)}+\widetilde{\chi^{(1)}(y)})^{2^{j-3}}  + 2^{j-2}\widetilde{\psi_{j-1}(x,y)} )^2$$
$$ = (\widetilde{\chi^{(1)}(x)}+\widetilde{\chi^{(1)}(y)})^{2^{j-2}} + 2^{j-1}(x^2+y^2)^{2^{j-3}}  \psi_{j-1}(x,y).$$
The result then follows from the congruence  $\psi_{j}(x,y)\equiv \psi_{j-1}(x,y) (x^2+y^2)^{2^{j-3}} \mod 2$.   
  %By (\ref{qadd}) it follows that  
%\begin{equation*}
%(\tx^2+\ty^2)^{2^{j-2}}\equiv (\chi^{(1)}(x)+\chi^{(1)}(y))^{2^{j-2}}\mod 2^{j-1}
%\end{equation*}
%which implies 
%\begin{equation}\label{eqn312}
%2^j\psi_{j}(\tx,\ty)(x^2+y^2)^{2^{j-2}}\equiv 2^j\psi_{j}(\tx,\ty)(\chi^{(1)}(x)+\chi^{(1)}(y))^{2^{j-2}}\mod 2^{j+1}
%\end{equation}
%Using the recursive formula for $\psi$, we have 
%$$
%2^{j-1}\psi_{j}(x,y)\equiv 2^{j-1}\psi_{j-1}(x,y)(x^2+y^2)^{2^{j-3}} +2^{2j-4}\psi_{j-1}(x,y)^2 
%$$
%$$
%\equiv
%2^{j-1}\psi_{j-1}(x,y) (\chi^{(1)}(x)+\chi^{(1)}(y))^{2^{j-3}}          \mod 2^{j}.
%$$
%The result then follows from our induction hypothesis. 

 We have shown that $\pi^i$ and $\chi^{(i)}$ for $0\le i\le m+1$ satisfy the conditions of Lemma \ref{lem} and, hence,  we have that $\phi_{m+1}$ is a homomorphism.

\subsection{} It remains to show that $\Phi_{m+1}$ is an isomorphism. As in the odd characteristic case, it suffices to check that 
\begin{equation}\label{eq21}
 r(Z_{m+1})=Im(F^{m+1}_{Z_0/R_0}),
 \end{equation}
  where $r:Z_{m+1}\to Z_0$ is the reduction map and $F^{m+1}_{Z_0/R_0}: Z_0^{(m+1)}\to Z_0$ is the $(m+1)^{th}$ iterate of the relative Frobenius.
  Let $\Pi:Z_{m}\to \Omega^1_{Z_0/R_0}$ be the map defined by formulas (\ref{eqpi1}), (\ref{eqpi2}).     
\begin{lm}
The image of an element $(z_1,\ldots,z_{m})\otimes a\in W_{m+1}(Z_0)\otimes R_{m}$ under the composition 
\[
S=\Pi\circ \Phi_{m}:W_{m+1}(z_0)\otimes R_m\rar{\sim}Z_{m}\to \Omega^1_{Z_0/R_0}
\]
 is given by the formula 
\begin{equation}
S((z_1,\ldots,z_{m+1})\otimes a)=\overline{a}\sum_{i=0}^{m}z_{i+1}^{2^{m-i}-1}dz_{i+1}
\end{equation}
where $\overline{a}\in R_0$ is the reduction of $a$ modulo $2$.
\end{lm}
\begin{proof}
We start with the following 
\begin{cl}
If $x\in Z_0$ and $z\in Z_i$ with $0< i\le m$ and $\tx,\tz\in A_{n}$ are any liftings we have 
\begin{equation}\label{eq30}
[\tz^2,\tx]\equiv 2\tz[\tz,\tx]\mod 2^{i+3}
\end{equation}
\end{cl}
Indeed we have 
\begin{align*}
[\tz^2,\tx]&=\tz[\tz,x]+[\tz,\tx]\tz\\
&=2\tz[\tz,\tx]+[\tz,[\tz,\tx]]
\end{align*}
Since $[\tz,\tx] \mod 2^{i+2}\in Z_{i+1}$ Lemma \ref{trivcomm} implies that $[\tz,[\tz,\tx]]\equiv 0\mod 2^{i+3}$ and the result follows. 

Formula (\ref{eq30}) implies, by induction,  that for $z\in Z_1$ and $i> 1$, we have 
\begin{equation}\label{2powcomm}
[\tz^{2^{i-1}},\tx]\equiv 2^{i-1}\tz^{2^{i-1}-1}[\tz,\tx]\mod 2^{i+2}.
\end{equation}

On the other hand, for $x, z\in Z_0$ using (\ref{qad}),  we have
\begin{align*}
[\tz^2+2\rz,\tx]&=[\tz^2,\tx]+2[\rz,\tx]\\
&= 2\tz[\tz,\tx]+[\tz,[\tz,\tx]]+2[\rz,\tx]\\
&\equiv 2\tz[\tz,\tx]+[\tz,[\tz,\tx]]+[\tz,[\tz,\tx]]\\
&\equiv 2\tz[\tz,\tx] \mod 8.
\end{align*}
This fact along with (\ref{2powcomm}) implies that for every  $x, z\in Z_0$, one has
\[
[(\tz^2+2\rz)^{2^{i-1}},\tx]\equiv  2^{i} (\tz^2+2\rz)^{2^{i-1}-1}  \tz [\tz,\tx]   = 2^{i+1} z^{2^{i}-1} \{z, x\} \mod 2^{i+2}.
\]
which proves the result.

\end{proof}
The above Lemma together with Lemma \ref{serre} imply (\ref{eq21}). Theorem \ref{2} is proven.

\section{Applications} Let  $S_n$ be a flat scheme over $\bZ/p^{n+1}\bZ$, and let 
$X_n \rar{f_n} S_n$ be a smooth scheme over  $S_n$. 
For $0\leq m\leq n$ we set  $X_m= X_n \times spec\,  \bZ/p^{m+1}\bZ \rar{f_m} S_n \times spec\,  \bZ/p^{m+1}\bZ = S_m $.  One has the relative Frobenius diagram.
\begin{equation}
\xymatrix{
X_0  \ar[dr]_{f_0}  \ar[r]^{F_{X_0/S_0}} & X^{(1)}_0  \ar[d]^{f'_0}  \ar[r]^{\pi_{X_0/S_0}} & X_0 \ar[d]^{f_0}\\
& S_0  \ar[r]^{F_S} & S_0\\
}
\end{equation}
Since the relative Frobenius morphism $F_{X_0/S_0}: X_0\to X^{(1)}_0$ is a homeomorphism the functor  $F_{X_0/S_0 *}$ induces an equivalence between the category of Zariski sheaves on $X_0$ and that on $X^{(1)}_0$.
We shall also identify the categories of Zariski sheaves on $X_n$ and on $X_0$.

We will write $D_{X_m/S_m}$ for the sheaf of PD differential operators on $X_m$ (\cite{bo}, \S 2) and $Z_m$ for its center. One has a canonical isomorphism of $f_0^{\prime  -1}\cO_{S_0}$-algebras on $X^{(1)}_0$
\begin{equation}\label{centrdif}
F_{X_0/S_0 *}(Z_0) \simeq S^{\cdot}T_{X^{(1)}_0/S_0}
 \end{equation}
 given by the p-curvarure map (see, {\it e.g.} \cite{ov}, Theorem 2.1). If $n>0$,  the construction from \S \ref{def} applied to affine charts of $f_n: X_n\to S_n$ yields  a  biderivation
 $$\{,\}: Z_0\otimes _{f_0^{-1}\cO_{S_0}} Z_0 \to Z_0,$$
that can be interpreted via isomorphism  (\ref{centrdif}) as a bivector field $$\mu \in \Gamma({\bf T}^*_{X^{(1)}_0/S_0}, \bigwedge^2 T_{{\bf T}^*_{X^{(1)}_0/S_0}/S_0})$$ on the cotangent space, ${\bf T}^*_{X^{(1)}_0/S_0}$.
On the other hand, the cotangent space to any smooth scheme has a canonical $1$-form (the ``contact form'') $$\eta_{can} \in \Gamma({\bf T}^*_{X^{(1)}_0/S_0},   \Omega^1_{{\bf T}^*_{X^{(1)}_0/S_0}/S_0})$$ whose differential $\omega_{can}=d\eta_{can}$ is a symplectic form. 
%We will denote
%by $\omega^{-1}$ the corresponding bivector field on ${\bf T}^*(X'_0/S_0)$.
\begin{lm}\label{konts} We have that $\mu^{-1} = - \omega_{can}$.
\end{lm}
\begin{proof}
This is proven in (\cite{kk}, Lemma 2) for $X_n= \bA^m_{S_n}$. The general case follows since the statement is local for \'etale topology.
 \end{proof}
If $\cA$ is a sheaf of algebras on a site $Y$ the presheaf $U\mapsto W_n(\cA(U))$ is a sheaf denoted by $W_n(\cA)$. 
\begin{Th}\label{3} There is a canonical isomorphism of sheaves of $f_0^{\prime -1}\cO_{S_n}$-algebras on $X^{(1)}_0$
 \begin{equation}\label{diffgen}
W_{n+1}(F_{X_0/S_0 *}(Z_0)) \otimes_{  W_{n+1}( f_0^{\prime -1}\cO_{S_0})} f_0^{\prime -1}\cO_{S_n} \iso  F_{X_0/S_0 *}(Z_n).
\end{equation} 
 \end{Th}
\begin{proof}
Let $U\subset X_n$ be open affine subset lying over an open affine subset $W=spec \, R_n \subset S_n$.
Then $R_n$ is flat algebra  over $\bZ/p^{n+1}\bZ$ and $A_n= \Gamma(U,  D_{X_n/S_n})$
is a flat algebra over $R_n$.  By (\ref{centrdif})  the center $Z(A_0)$ of its reduction modulo $p$ 
is isomorphic to $$S^{\cdot}T_{\cO(U)/R_0} \otimes _{F_{R_0}} R_0$$ which is smooth over $R_0$.
Moreover, Lemma \ref{konts} shows that the deformation  $A_n$ is non-degenerate and the associated $2$-form is the differential of a canonical $1$-form, $-\eta_{can}$.
Thus, by Theorems \ref{p} and \ref{2} we get a canonical isomorphism
$$\Phi_n: W_{n+1}(Z(A_0))\otimes _{W_{n+1}(R_0)} R_n \iso Z(A_{n})$$
There exists a unique isomorphism of sheaves of algebras (\ref{diffgen}) that induces $\Phi_n$  
for each pair $U,W$ as above. 
\end{proof}
Combing (\ref{diffgen}) with (\ref{centrdif}) we find an isomorphism
$$W_{n+1}(S^{\cdot}T_{X^{(1)}_0/S_0}) \otimes_{  W_{n+1}( f_0^{\prime -1}\cO_{S_0})} f_0^{\prime -1}\cO_{S_n} \iso  F_{X_0/S_0 *}(Z_n).$$

\begin{rem} The fact that the center of $D_{X_m/S_m}$ depends only on $X_0 \to S_n$ and not on the deformation $X_n \to S_n$ is not surprising:  the category of  quasi-coherent $D_{X_m/S_m}$-modules on $X_m$
is equivalent to the category of quasi-coherent crystals on $X_0/S_n$ (\cite{bo}). In particular its categorical center\footnote{Recall, that the center of a category $\cA$ is the ring of endomorphisms of the identity functor $Id_\cA: \cA\to \cA$.}, which is just $Z_n$, is isomorphic to the center of the category of crystals. 
\end{rem}

%Thus, the global section of $D_{X_n/S_n}$ on an open affine  subset $U\subset X_n$ is the free $\bZ/p^{n+1}\bZ$-algebra generated by functions, $g\in \cO(U)$ and derivations, $\alpha\in \Gamma(U,T_{X_n/S_n})$,
%subject to the relations 
%\[
%f\cdot \alpha=f\alpha, \qquad \alpha\cdot f-f\cdot \alpha=\alpha(f), \qquad \alpha_1\cdot \alpha_2 - \alpha_2\cdot \alpha_1=[\alpha_1, \alpha_2],
%\]
%where $[\alpha_1, \alpha_2]\in \Gamma(U,T_{X_n/S_n})$ denotes the commutator of vector fields.

\end{document}